\documentclass{amsart}
\usepackage{amsmath,amssymb,amsthm,amsfonts,amscd,mathrsfs,mathtools}
\usepackage{tikz,graphicx}
\usepackage[all]{xy}
\usepackage[hyphens]{url}
\usepackage{hyperref}
\usepackage{caption}
\usetikzlibrary{decorations.markings}

\theoremstyle{plain}
    \newtheorem{thm}{Theorem}[section]
    \newtheorem{lem}[thm]   {Lemma}
    \newtheorem{cor}[thm]   {Corollary}
    \newtheorem{prop}[thm]  {Proposition}

\theoremstyle{definition}
    \newtheorem{defn}[thm]  {Definition}

    \newtheorem{ex}[thm]{Example}
    \newtheorem{rem}[thm]{Remark}

\def\secat{\mathsf{secat}}

\def\dim{\mathrm{dim}}

\newcommand{\be}{\begin{enumerate}}
\newcommand{\ee}{\end{enumerate}}
\newcommand{\R}{\mathbb{R}}
\newcommand{\Z}{\mathbb{Z}}

\newcommand{\TC}{{\sf TC}}

\begin{document}

\title[Symmetrized TC]{Symmetrized topological complexity}

\author{Mark Grant}

\address{Institute of Mathematics,
Fraser Noble Building,
University of Aberdeen,
Aberdeen AB24 3UE,
UK}

\email{mark.grant@abdn.ac.uk}

\date{\today}


\begin{abstract} We present upper and lower bounds for symmetrized topological complexity $\TC^\Sigma(X)$ in the sense of Basabe--Gonz\'alez--Rudyak--Tamaki. The upper bound comes from equivariant obstruction theory, and the lower bounds from the cohomology of the symmetric square $SP^2(X)$. We also show that symmetrized topological complexity coincides with its monoidal version, where the path from a point to itself is required to be constant. Using these results, we calculate the symmetrized topological complexity of all odd spheres.
\end{abstract}


\maketitle
\section{Introduction}\label{sec:intro}

Topological complexity is a numerical homotopy invariant defined by Farber as part of his topological approach to the robot motion planning problem \cite{Far03,Far04}. It can be defined as the sectional category (or Schwarz genus \cite{Schwarz}) of the free path fibration. More precisely, let $X$ be a path-connected space, and let $PX$ denote the space of paths in $X$ with the compact-open topology. The evaluation map
\[
\pi:PX\to X\times X,\qquad \pi(\gamma)=\left(\gamma(0),\gamma(1)\right)
\]
is a fibration, which admits a (global) section if and only if $X$ is contractible. The topological complexity of $X$, denoted $\TC(X)$, is defined to be the minimal integer $k$ such that $X\times X$ admits a cover $U_0, U_1,\ldots , U_k$ by open sets, on each of which $\pi$ admits a local section $\sigma_i:U_i\to PX$. (Note that we employ the normalization convention that categorical invariants are one less than the number of sets in the cover.)

The local sections $\sigma_i:U_i\to PX$ appearing in the definition are called motion planners, because they assign to each $(A,B)\in U_i$ a path in $X$ from $A$ to $B$. One might wish to impose additional, natural constraints on the motion planners, such as that the motion from $A$ to $A$ be constant, or that the motion from $B$ to $A$ be the reverse of the motion from $A$ to $B$. This motivates the definition of \emph{symmetric topological complexity}, given by Farber and Grant in \cite{FG06}. Removing the diagonal $\Delta X\subseteq X\times X$ and its pre-image under $\pi$, one obtains a fibration
\[
\pi':P'X\to F(X,2),
\]
where $P'X$ (the space of paths with distinct endpoints) and $F(X,2)$ (the two-point ordered configuration space on $X$) each carry a free $\Z/2$ action---one by reversing paths, the other by transposition. On passing to orbit spaces one obtains a fibration
\[
\pi'':P'X/\Z/2\to B(X,2)
\]
over the unordered configuration space. Then the symmetric topological complexity $\TC^S(X)$ is defined to be the sectional category of $\pi''$, plus one (to account for the points on the diagonal, where the motion planner must remain constant).

There are admittedly several drawbacks to this definition. Firstly, it is not clear that treating (a neighbourhood of) the diagonal as a subset on its own is the most efficient way to motion plan, even when motions from a point to itself are required to be constant. Secondly, there is the intrinsic difficulty of working with the configuration spaces $B(X,2)$, whose cohomology is not yet fully understood. Thirdly, and perhaps most significantly, $\TC^S(X)$ is not a homotopy invariant---although the only known example illustrating this fact is the rather artificial case of a contractible space which is not a point (see \cite[Example 7]{FG06}).

Many of these difficulties are overcome if one uses a modified definition due to Basabe, Gonz\'alez, Rudyak and Tamaki \cite{BGRT}. Namely, consider $\pi:PX\to X\times X$ as a $\Z/2$-equivariant map, and let $\TC^\Sigma(X)$ denote the minimal integer $k$ such that $X\times X$ admits a cover $U_0, U_1,\ldots , U_k$ by invariant open sets, on each of which $\pi$ admits an equivariant local section. (The authors of \cite{BGRT} in fact define a higher version of this invariant, associated with the iterated diagonal $\Delta_m:X\to X^m$; we focus mainly on the $m=2$ case of their definition.) Following a suggestion of J.\ Gonz\'alez, we refer to this invariant as \emph{symmetrized topological complexity}, reserving the name symmetric topological complexity for the Farber--Grant version. It is immediate that $\TC(X)\le \TC^\Sigma(X)$, and straightforward to show that $\TC^\Sigma(X)$ is homotopy invariant. With a little more work, it can be shown \cite[Proposition 4.2]{BGRT} that
\[
\TC^S(X)-1\le \TC^\Sigma(X)\le \TC^S(X).
\]
The authors of \cite{BGRT} left open the question of whether the first inequality can be an equality for non-contractible spaces, or the determination of $\TC^\Sigma(X)$ for well-known manifolds, notably odd spheres.

The goal of the current paper is to begin a systematic study of $\TC^\Sigma(X)$, by placing it in the framework of equivariant sectional category (as introduced by Colman and Grant in \cite{CG}). This allows the use of equivariant obstruction theory to give the following upper bound, entirely analogous to the upper bound for ordinary topological complexity given in \cite[Theorem 5.2]{Far04}.

\newtheorem*{thm:upper}{Theorem \ref{thm:upper}}
\begin{thm:upper}
Let $X$ be an $s$-connected polyhedron. Then
\[
\TC^\Sigma(X)< \frac{2\,\dim\,X+1}{s+1}.
\]
\end{thm:upper}

One potential perceived drawback of the definition of $\TC^\Sigma(X)$ is that it lacks the condition that the motion from $A$ to $A$ be constant. If we add this condition in, we arrive at a definition of \emph{monoidal symmetrized topological complexity}, denoted $\TC^{M,\Sigma}(X)$, see Definition \ref{def:MSTC} below. This is a symmetrized version of the monoidal topological complexity $\TC^M(X)$ of Iwase and Sakai \cite{IS1}. Those authors conjectured that $\TC^M(X)=\TC(X)$ for locally finite simplicial complexes (see \cite{IS2}; the paper \cite{IS1} presents a proof which was found to contain an error). In the symmetrized case we are able to show that the two notions coincide on a large class of spaces.

\newtheorem*{thm:MSvsS}{Theorem \ref{thm:MSvsS}}
\begin{thm:MSvsS}
Let $X$ be a paracompact ENR. Then $\TC^{M,\Sigma}(X)=\TC^\Sigma(X)$.
\end{thm:MSvsS}

Our motivation for introducing $\TC^{M,\Sigma}(X)$ and proving Theorem \ref{thm:MSvsS} is to give cohomological lower bounds for $\TC^\Sigma(X)$ in terms of relative cohomology. Recall that the symmetric square $SP^2(X)$ of $X$ is the orbit space of $X\times X$ under the involution given by transposing coordinates. Let $dX\subseteq SP^2(X)$ denote the image of the diagonal $\Delta X\subseteq X\times X$.

\newtheorem*{thm:lowerSP}{Theorem \ref{thm:lowerSP}}
\begin{thm:lowerSP}
 Suppose there are classes $x_1,\ldots ,x_k\in H^*(SP^2(X))$ (with arbitrary coefficients) such that:
 \be
 \item $x_i$ restricts to zero in $H^*(dX)$ for $i=1,\ldots, k$;
 \item $0\neq x_1\cdots x_k \in H^*(SP^2(X))$.
 \ee
 Then $\TC^\Sigma(X)\ge k$.
 \end{thm:lowerSP}

\newtheorem*{thm:lowerMonoidal}{Theorem \ref{thm:lowerMonoidal}}
\begin{thm:lowerMonoidal}
Suppose there are relative classes $x_1,\ldots ,x_k\in H^*(SP^2(X),dX)$ (with arbitrary coefficients) such that
$$0\neq x_1\cdots x_k \in H^*(SP^2(X),dX).$$ Then $\TC^{M,\Sigma}(X)\ge k$.
 \end{thm:lowerMonoidal}

 Our test spaces for these results are odd spheres. Nakaoka has made extensive computations of the cohomology rings of cyclic and symmetric powers (see \cite{Nak56} for instance), particularly of spheres. The symmetric square $SP^2(S^1)$ is homeomorphic to a (compact) M\"obius band. Letting $X=S^n$ and applying Theorem \ref{thm:lowerSP} (in the case $n>1$) and Theorem \ref{thm:lowerMonoidal} (in the case $n=1$), we obtain the following result.

 \newtheorem*{thm:spheres}{Theorem \ref{thm:spheres}}
 \begin{thm:spheres}
 We have $\TC^\Sigma(S^n)=2$ for all $n\ge 1$.
 \end{thm:spheres}

 We remark that the case $n$ even is already contained in \cite[Example 4.5]{BGRT}. Also, Don Davis \cite{Davis} has recently proved that $\TC^\Sigma(S^1)=2$ by elementary methods (using results from general topology).

 It should be noted that other cohomological lower bounds for $\TC^\Sigma(X)$ are available, using various flavours of equivariant cohomology. Indeed, if $h^*$ is any $\Z/2$-equivariant cohomology theory with products, then the cup-length of the kernel of $h^*(X\times X)\to h^*(\Delta X)$ provides a lower bound for $\TC^\Sigma(X)$. We have found ordinary Borel equivariant cohomology (with untwisted coefficients) insufficient for proving Theorem \ref{thm:spheres}. In theory, Bredon cohomology should be the right tool for proving lower bounds (as upper bounds); in practice, however, the difficulty of computing cup products in Bredon cohomology renders this impractical. In the end, the cohomology of symmetric squares proved to be the right setting for proving Theorem \ref{thm:spheres}, and more besides. For instance, Jes\'us Gonz\'alez has used calculations in the cohomology of symmetric squares of real projective spaces to calculate $\TC^\Sigma(\R P^m)$ when $m$ is a $2$-power \cite{Gonzalez}.

 The contents of the paper are as follows. In Section \ref{sec:prelim} we collect several results about $G$-fibrations and $G$-cofibrations which we will need in the sequel. In Section \ref{sec:Gsecat} we generalize various results of Schwarz \cite{Schwarz} to the equivariant setting, including a general dimension-connectivity upper bound for equivariant sectional category which we employ in Section \ref{sec:SymmTC} to prove Theorem \ref{thm:upper}. In the same section we prove the lower bound Theorem \ref{thm:lowerSP}. Section \ref{sec:MonoidalSymmTC} introduces monoidal symmetrized topological complexity and proves Theorems \ref{thm:MSvsS} and \ref{thm:lowerMonoidal}. The preceding results are applied in Section \ref{sec:spheres} to calculate symmetrized topological complexity of spheres. Finally, in Section \ref{sec:higher} we indicate which of our results generalize to symmetrized higher topological complexity, leaving open the question of determining the symmetrized higher topological complexity of odd spheres.

 We wish to thank Don Davis and Jes\'us Gonz\'alez for useful discussions and for sharing with us preliminary versions of their articles \cite{Davis} and \cite{Gonzalez}. We also thank Yuli Rudyak for suggesting we include higher versions of our results.

\section{\texorpdfstring{Preliminaries on $G$-fibrations and $G$-cofibrations}{Preliminaries on G-fibrations and G-cofibrations}}\label{sec:prelim}

Here we recall several facts about $G$-fibrations and $G$-cofibrations. Often the proofs are straight-forward generalizations of the corresponding non-equivariant results, hence are omitted. We work in a category of spaces convenient for doing homotopy theory, such as compactly-generated weak Hausdorff spaces.

In this section $G$ will denote a general topological group, and we use standard terminology and notation from equivariant topology, such as $G$-space, $G$-map, $G$-homotopy and so forth. For instance if $H\le G$ is a closed subgroup, then a $G$-map $f:X\to Y$ induces a map $f^H:X^H\to Y^H$ of $H$-fixed point spaces.

By a \emph{(Serre) $G$-fibration} we will mean a $G$-map $p:E\to B$ having the $G$-homotopy lifting property with respect to all $G$-spaces ($G$-CW complexes).

\begin{lem}[{Compare \cite[\S III.4]{Bredon2}, \cite[15.3]{Lueck}}]\label{lem:SerreGfibr}
If $G$ is a compact Lie group, then $p:E\to B$ is a Serre $G$-fibration if and only if $p^H:E^H\to B^H$ is a Serre fibration for all subgroups $H\le G$.
\end{lem}

Given a $G$-space $B$, there is a natural action of $G$ on the space $PB$ of paths in $B$ (with the compact-open topology). The evaluation map $\mathrm{ev}_0:PB\to B$ given by $\mathrm{ev}_0(\gamma)=\gamma(0)$ is equivariant. For any $G$-map $p:E\to B$, the pullback
\[
E\times_B PB = \{(e,\gamma)\in E\times PB \mid p(e)=\gamma(0)\}
\]
is again a $G$-space in a natural way. Define a \emph{$G$-lifting function} for $p:E\to B$ to be a $G$-map $\lambda_p:E\times_B PB\to PE$ such that $p\circ \lambda_p(e,\gamma)=\gamma$ for all $(e,\gamma)\in E\times_B PB$.

\begin{lem}[{Compare \cite[Theorem 2.7.8]{Spa}}]\label{lem:HurGfibr}
A $G$-map $p:E\to B$ is a $G$-fibration if and only if it admits a $G$-lifting function.
\end{lem}

 By a \emph{$G$-cofibration} we will mean an inclusion of $G$-spaces $i:A\hookrightarrow X$ satisfying the $G$-homotopy extension property. The next result gives several alternative characterisations of $G$-cofibrations. Recall that a pair $(X,A)$ of $G$-spaces is called a \emph{$G$-NDR pair} if there exists a $G$-neighbourhood $U$ of $A$ in $X$ which is deformable rel $A$ into $A$ (meaning there exists a $G$-homotopy $H:U\times I\to X$ satisfying $H(x,0)=x$, $H(a,t)=a$ and $H(x,1)\in A$ for all $x\in U$, $a\in A$, $t\in I$), and a $G$-map $\varphi:X\to I$ such that $\varphi^{-1}(0)=A$ and $\varphi(X-U)=1$.

\begin{lem}[Compare \cite{Strom}]\label{lem:Gcofibr}
Let $i:A\hookrightarrow X$ be an inclusion of $G$-spaces. The following are equivalent:
\begin{enumerate}
\item $i$ is a $G$-cofibration;
\item The pair $(X,A)$ is a $G$-NDR pair;
\item There exists an equivariant retraction $r:X\times I \to X\times\{0\}\cup A\times I$, where $G$ acts trivially on $I$ and diagonally on $X\times I$.
\end{enumerate}
\end{lem}

Note that if $i:A\hookrightarrow X$ is a $G$-cofibration and $X$ is Hausdorff, then $A\subseteq X$ is closed.

\begin{ex}\label{ex:interval}
Let $\hat I$ denote the interval $[0,1]$ with the involution $t\mapsto 1-t$. Then the inclusion $i:\{0,1\}\hookrightarrow \hat I$ is a $\Z/2$-cofibration. An equivariant retraction $r:\hat I\times I\to \hat I\times \{0\}\cup \{0,1\}\times I$ is depicted in Figure \ref{fig:Icofibr}.

\begin{figure}
\begin{tikzpicture}
\draw[fill=gray!60] (0,-1) rectangle (2,1);
\draw[ultra thick] (2,1)--(0,1)--(0,-1)--(2,-1);
\draw[fill] (4,0) circle [radius=0.05];
\draw[dashed] (4,0)--(1.6,1);
\draw[dashed] (4,0)--(0.4,1);
\draw[dashed] (4,0)--(0,.6);
\draw[dashed] (4,0)--(0,0);
\draw[dashed] (4,0)--(0,-.6);
\draw[dashed] (4,0)--(0.4,-1);
\draw[dashed] (4,0)--(1.6,-1);
\draw[ultra thick, <->] (-.2,-.7) to [out=120,in=240] (-.2,.7);
\end{tikzpicture}
\caption{An equivariant retraction $r:\hat I\times I\to \hat{I}\times\{0\}\cup\{0,1\}\times I$ is obtained by radial projection}
\label{fig:Icofibr}
\end{figure}
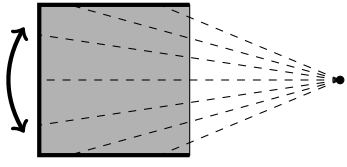

\end{ex}

If $X$ is any $G$-space and $Y$ is any space, then the mapping space $Y^X$ (given the compact-open topology) inherits an action of $G$ via pre-composition.

\begin{prop}[{Compare \cite[Theorem 2.8.2]{Spa}}]\label{prop:Spanier}
Let $i:A\to X$ be a $G$-cofibration, where $A$ and $X$ are locally compact Hausdorff spaces, and let $Y$ be any space. Then the induced map
\[
i^*:Y^X\to Y^A,\qquad i^*(f)= f\circ i
\]
is a $G$-fibration.
\end{prop}

\begin{ex}\label{ex:Gfibration}
For any space $X$, let $PX=X^I$ denote the space of all paths in $X$. By Example \ref{ex:interval} and Proposition \ref{prop:Spanier}, the evaluation map
\[
\pi:PX\to X\times X,\qquad \pi(\gamma)=\big(\gamma(0),\gamma(1)\big)
\]
is a $\Z/2$-fibration, where $\Z/2$ acts on $X\times X$ by transposition of factors and on $PX$ by reversal of paths.

For later use, we note that restricting to fixed points yields a fibration over the diagonal $(X\times X)^{\Z/2}=\Delta X$ whose total space $(PX)^{\Z/2}=\{\gamma\in PX \mid \overline{\gamma}=\gamma\}$ is homeomorphic to $PX$ (via a homeomorphism which restricts a symmetric path to its first half). In fact, this fibration $\pi^{\Z/2}$ is homeomorphic with the fibration
 \[
 \xymatrix{
 P_0X \ar[r] & PX \ar[r]^{\mathrm{ev}_0} & X
 }
 \]
 given by evaluation at $0$, with fibre the based path space. Hence $\pi^{\Z/2}$ is a homotopy equivalence.
\end{ex}

Recall that a $G$-space $X$ is a \emph{$G$-ENR} if it embeds as a $G$-retract of an open neighbourhood $U$ in some finite-dimensional $G$-representation $V$.

\begin{prop}\label{prop:GLEC}
Let $X$ be a $G$-ENR, and let $A\subseteq X$ be a closed sub-$G$-ENR. Then the inclusion $i:A\hookrightarrow X$ is a $G$-cofibration.
\end{prop}
\begin{proof}
By Lemma \ref{lem:Gcofibr} it suffices to show that $(X,A)$ is a $G$-NDR pair. Let $W$ be an invariant neighbourhood of $A$ in $X$ with a $G$-retraction $r:W\to A$. Let $W'$ be an invariant neighbourhood of $X$ in some $G$-representation $V$ with a $G$-retraction $s:W'\to X$. Define
\[
U= \{ x\in W \mid (1-t)x + tr(x) \in W'\mbox{ for all }t\in [0,1]\}.
\]
Now define a $G$-homotopy $H:U\times I\to X$ by $H(x,t)=s\left((1-t)x + tr(x)\right)$. Finally, since $X$ is $G$-perfectly normal, we may take a $G$-map $\varphi:X\to I$ such that $\varphi^{-1}(0)=A$ and $\varphi(X-U)=1$.
\end{proof}

\begin{cor}\label{cor:symLEC}
Let $X$ be an ENR. Then the diagonal inclusion $\Delta: X\hookrightarrow X\times X$ is a $\Z/2$-cofibration, where $\Z/2$ acts trivially on $X$ and by transposition on $X\times X$.
\end{cor}

Spaces satisfying the conclusion of Corollary \ref{cor:symLEC} will be called \emph{symmetrically locally equi-connected}, or \emph{symmetrically LEC}.

\begin{prop}[{Relative $G$-homotopy lifting (compare \cite{Strom}, \cite[15.5]{Lueck})}]\label{prop:relGHLP}
Let $p:E\to B$ be a $G$-fibration, and let $i:A\to X$ be a closed $G$-cofibration. Then the relative $G$-homotopy lifting problem
\[
\xymatrix{
X\times \{0\}\cup A\times I \ar[rr]^-{\tilde{H_0}\cup K} \ar[d] & &E \ar[d]^p \\
X\times I\ar[rr]_{H} \ar@{.>}[urr]^{F} & & B
}
\]
has a solution $F:X\times I\to E$ for all initial data $H,\tilde{H_0},K$.
\end{prop}

\section{Equivariant sectional category}\label{sec:Gsecat}

In this section $G$ will denote a compact Lie group. We will develop some properties of the equivariant sectional category of a $G$-fibration, introduced in \cite{CG}. These will be applied in later sections to study the symmetrized topological complexity.


\begin{defn}
The \emph{equivariant sectional category} of a $G$-fibration $p:E \to B$, denoted
$\secat_G(p)$, is the least integer $k$ such that $B$ may be covered by invariant open
sets $U_0,U_1,\ldots , U_k$ on each of which there exists a local $G$-section of $p$, that is, a $G$-map $\sigma_i: U_i\to E$ such that
$p\sigma_i=\mathrm{incl}:U_i\hookrightarrow B$.
\end{defn}

More generally, the equivariant sectional category of any $G$-map $p:E\to B$ may be defined in terms of local $G$-homotopy sections of $p$, as in \cite[Definition 4.1]{CG}. It is immediate from the $G$-HLP that the two definitions coincide for $G$-fibrations.

We are going to generalize various results of Schwarz \cite{Schwarz} about the ordinary sectional category $\secat(p)$ to the equivariant setting. Recall that a $G$-space $X$ is called \emph{$G$-paracompact} if for every cover $\mathscr{U}=\{U_\lambda\}_{\lambda\in \Lambda}$ of $X$ by invariant open sets, there exists a \emph{$G$-equivariant partition of unity} subordinate to $\mathscr{U}$ (that is, a collection of $G$-maps $\{h_\lambda:X\to I\}_{\lambda\in \Lambda}$ from $X$ to the interval $I=[0,1]$ endowed with the trivial $G$-action which forms a partition of unity subordinate to $\mathscr{U}$ in the usual sense).

\begin{lem}\label{lem:Gparacompact}
Let $X$ be a paracompact $G$-space. Then $X$ is $G$-paracompact.
\end{lem}

\begin{proof}
Given a cover $\mathscr{U}=\{U_\lambda\}$ of $X$ by $G$-invariant open subsets, we may find a non-equivariant partition of unity $\{g_\lambda\}$ subordinate to $\mathscr{U}$. Now obtain a $G$-equivariant partition of unity $\{h_\lambda\}$ subordinate to $\mathscr{U}$ by integration:
\[
h_\lambda(x) = \frac{1}{\mu(G)}\int_{a\in G} g_\lambda(ax)\, d\mu,
\]
where $\mu$ is the Haar measure on $G$.
\end{proof}

Let $p:E\to B$ be any map, and let $k$ be a non-negative integer. Then the \emph{$(k+1)$-fold fibred join} of $p$ is defined as follows. Let
\[
J^k_B(E):=\{ (e_0,t_0,\ldots, e_k,t_k) \mid e_i\in E, t_i\in [0,1],p(e_i)=p(e_j),\sum t_i =1\}/\sim,
\]
where the equivalence relation $\sim$ is generated by $$(e_0,t_0,\ldots, e_i,0,\ldots e_k,t_k)\sim (e_0,t_0,\ldots, e_i',0,\ldots e_k,t_k).$$ Define a map $p_k:J^k_B(E)\to B$ by setting
 \[
 p_k\big([e_0,t_0,\ldots, e_k,t_k]\big) = p(e_0)=\cdots = p(e_k).
 \]
 If $p:E\to B$ is a $G$-map, then $p_k:J^k_B(E)\to B$ is also a $G$-map, where $J^k_B(E)$ is given the diagonal $G$-action
\[
g\cdot [e_0,t_0,\ldots, e_k,t_k]=[g\cdot e_0,t_0,\ldots, g\cdot e_k,t_k].
\]

\begin{lem}
If $p:E\to B$ is a (Serre) $G$-fibration with fibre $F$, then $p_k:J^k_B(E)\to B$ is a (Serre) $G$-fibration with fibre $J^k(F)$, the $(k+1)$-fold iterated join of $F$.
\end{lem}

\begin{proof}
If $p:E\to B$ is a $G$-fibration, then by Lemma \ref{lem:HurGfibr} it admits a $G$-lifting function $\lambda_p:E\times_B PB\to PE$. We may then construct a $G$-lifting function $\lambda_{p_k}:J^k_B(E)\times_B PB\to PJ^k_B(E)$ for $p_k$ by setting
\[
\lambda_{p_k}\left([e_0,t_0,\ldots, e_k,t_k],\gamma\right) = [\lambda_p(e_0,\gamma),t_0,\ldots , \lambda_p(e_k,\gamma),t_k].
\]
Hence $p_k$ is a $G$-fibration.

For the statement about Serre $G$-fibrations, we use Lemma \ref{lem:SerreGfibr}. It suffices to check that $p_k^H:J^k_B(E)^H\to B^H$ is a Serre fibration for all subgroups $H\le G$. Clearly taking fixed points commutes with fibred joins, so that this map is homeomorphic to $(p^H)_k:J^k_{B^H}(E^H)\to B^H$. The result now follows since the fibred join of (Serre) fibrations is again a (Serre) fibration.
\end{proof}

Just as in the non-equivariant case, the fibred join construction allows us to reduce the computation of equivariant sectional category to the question of existence of sections of auxiliary fibrations.

\begin{prop}\label{prop:Gfibjoin}
Let $p:E\to B$ be a $G$-fibration over a paracompact base space. Then $\secat_G(p)\le k$ if and only if $p_k:J_B^k(E)\to B$ admits a (global) $G$-section.
\end{prop}

\begin{proof}
The original proof of Schwarz \cite[Theorem 3]{Schwarz} can be made equivariant, however the proof given by James in \cite[Proposition 8.1]{James} is simpler.

Suppose $p_k$ admits a $G$-section $\sigma:B\to J^k_B(E)$. We may think of the join coordinates in $J^k_B(E)$ as continuous functions $t_i:J^k_B(E)\to [0,1]$. Define $U_i:=\sigma^{-1}(t_i^{-1}(0,1])$. Hence $U_i$ consists of those $b\in B$ such that the $i$-th join coordinate of $\sigma(b)$ is positive. The $U_i$ are open in $B$, and clearly invariant. We may define local $G$-sections $\sigma_i:U_i\to E$ by the formula
\[
\sigma(b) = [\sigma_0(b),t_0,\ldots , \sigma_k(b),t_k].
\]

Conversely, suppose $\secat_G(p)\le k$ and we have a cover of $B$ by invariant open sets $U_0,\ldots , U_k$, each of which admits a local $G$-section $\sigma_i:U_i\to E$. Since $B$ is $G$-paracompact by Lemma \ref{lem:Gparacompact}, we may take a $G$-equivariant partition of unity $\{h_0,\ldots ,h_k\}$ subordinate to this cover. Then
\[
\sigma(b) = [\sigma_0(b),h_0(b),\ldots , \sigma_k(b),h_k(b)]
\]
defines a $G$-section $\sigma:B\to J^k_B(E)$ of $p_k$.
\end{proof}

The question of existence of equivariant sections of $G$-fibrations is addressed by equivariant obstruction theory. The basic theory (for $G$ finite) was outlined by Bredon in  \cite{Bredon1,Bredon2}, including the definitions of the Bredon cohomology groups $H^*_G(X;M)$, which form the natural home for equivariant obstruction classes. These references only consider the equivariant extension problem, however. For full details of the equivariant lifting problem for $G$-fibrations (where $G$ is compact Lie), we refer the reader to \cite{MukMuk}. Once the theory is in place, the following is a generalization of \cite[Theorem 5]{Schwarz}.

\begin{thm}\label{thm:uppersecatG}
Let $p:E\to B$ be a Serre $G$-fibration with fibre $F$, whose base $B$ is a $G$-CW complex of dimension at least $2$. Assume that $\pi_n(F^H)=0$ for all subgroups $H\le G$ and all $n<s$, where $s\ge 0$. Then $$\secat_G(p)< \frac{\dim\, B+1}{s+1}. $$
\end{thm}

\begin{proof}
By Proposition \ref{prop:Gfibjoin} it suffices to show that $p_k:J^k_B(E)\to B$ admits a section whenever $k\ge \frac{\dim B+1}{s+1}-1$. The obstructions to such a section live in Bredon cohomology groups
\[
H^{n+1}_G\big(B;\pi_n(J^k\mathscr{F})\big),
\]
where $\pi_n(J^k\mathscr{F})$ denotes an equivariant local coefficient system on $B$, the values of which on $B^H$ are isomorphic to $\pi_n(J^k(F)^H)=\pi_n(J^k(F^H))$. By our assumption that $F^H$ is $(s-1)$-connected, the $(k+1)$-fold join $J^k(F^H)$ is $(s+1)(k+1)-2$-connected. Hence the obstructions all lie in zero groups, provided
\[
\dim\, B - 1 \le (s+1)(k+1)-2  \iff \dfrac{\dim\, B + 1}{s+1} \le k+1.
\]
(The assumption $\dim(B)-1\ge 1$ ensures that the spaces $J^k(F^H)$ are all simply-connected, and in particular $n$-simple.)
\end{proof}

\section{Symmetrized topological complexity}\label{sec:SymmTC}

For the remainder of the paper, $G$ will denote the group $\Z/2$. For any space $X$, let $PX$ denote the space of all paths in $X$, given the compact-open topology. The evaluation map
\[
\pi:PX\to X\times X,\qquad \pi(\gamma)=\big(\gamma(0),\gamma(1)\big)
\]
is $G$-equivariant with respect to the involutions given by reversing paths and transposing coordinates. We have seen in Example \ref{ex:Gfibration} that $\pi$ is a $G$-fibration.

\begin{defn}[\cite{BGRT}]
The \emph{symmetrized topological complexity} of $X$ is
\[
\TC^\Sigma(X):=\secat_G(\pi).
\]
\end{defn}

Our first result is an upper bound for $\TC^\Sigma(X)$ analogous to the upper bound for $\TC(X)$ given by Farber in \cite[Theorem 5.2]{Far04}.

\begin{thm}\label{thm:upper}
Let $X$ be an $s$-connected polyhedron. Then
\[
\TC^\Sigma(X)< \frac{2\,\dim\,X+1}{s+1}.
\]
\end{thm}

\begin{proof}
Suppose $X\approx |K|$ for some simplicial complex $K$. The simplicial Cartesian product $K\times K$ is defined in \cite[Chapter II.8]{EilenbergSteenrod}; it is a simplicial complex such that $X\times X\approx |K\times K|$. Note that $G$ acts simplicially on $K\times K$, with fixed points given by the diagonal subcomplex $\Delta K\subseteq K\times K$. This gives $X\times X$ the structure of a $G$-CW complex.

Therefore we may apply Theorem \ref{thm:uppersecatG} to the $G$-fibration $\pi:PX\to X\times X$. Note that the fibre $\Omega X$ is $(s-1)$-connected. The fixed point subspace $(\Omega X)^G=\{\gamma\in \Omega X \mid \overline\gamma=\gamma\}$ is homeomorphic to the based path space $P_0X$ (as noted in Example \ref{ex:Gfibration}), hence is contractible. The result follows immediately.
\end{proof}

\begin{rem}
As was observed in Section 6 of \cite{GonLan}, one can apply the non-equivariant \cite[Theorem 5]{Schwarz} to the fibration $\pi'':P'(X)/\Z/2\to B(X,2)$ to get an upper bound for $\TC^S(X)$. When $X=M$, an $s$-connected closed smooth manifold, the configuration space $B(M,2)$ has the homotopy type of a complex of dimension $2\,\dim M-1$ and therefore
\[
\TC^\Sigma(M)\le \TC^S(M)< \frac{2\,\dim\, M}{s+1} + 1 = \frac{2\,\dim \, M +1}{s+1} + \frac{s}{s+1}.
\]
This upper bound for $\TC^\Sigma(M)$ is improved on by Theorem \ref{thm:upper}, as long as $2\,\dim(M)$ is not a multiple of $s+1$. This suggests that when looking for spaces with $\TC^\Sigma(X)<\TC^S(X)$, one might look at $2$-connected manifolds, or non-manifold polyhedra. The difficulty in finding such examples is that the known lower bounds for $\TC^S(X)$ from \cite{FG06} are valid only for manifolds, and are anyway also lower bounds for $\TC^\Sigma(X)$.
\end{rem}

\begin{cor}
Let $X$ be any $1$-connected, closed symplectic manifold. Then $$\TC^\Sigma(X) = \dim\,X.$$
\end{cor}
\begin{proof}
The lower bound follows from \cite[Corollary 3.2]{FTY} which gives
\[
\dim\, X = \TC(X)\le \TC^\Sigma(X),
\]
and the upper bound follows from Theorem \ref{thm:upper}.
\end{proof}

We next describe a cohomological lower bound for $\TC^\Sigma(X)$ in terms of the symmetric square construction.

\begin{lem}[{Compare \cite[Lemma 18.1]{Far06}}]\label{lem:sectioniffdiagonal}
An invariant open subset $U\subseteq X\times X$ admits a local $G$-section of $\pi$ if and only if the inclusion $U\hookrightarrow X\times X$ is $G$-homotopic to a map with values in the diagonal $\Delta X\subseteq X\times X$.
\end{lem}

\begin{proof}
Let $H:U\times I\to X\times X$ be a $G$-homotopy satisfying $H(x,y,0)=(x,y)$ and $H(x,y,1)\in \Delta X$ for all $(x,y)\in U$. Then we may define a local $G$-section $\sigma:U\to X^I$ by setting
\[
\sigma(x,y)(t)= \left\{\begin{array}{ll} p_1 H(x,y,2t), & 0\le t\le \frac12 \\ p_2 H(x,y,2-2t), & \frac12\le t\le 1 \end{array}\right.
\]
where $p_i: X\times X\to X$ for $i=1,2$ denotes projection onto the $i$-th coordinate. One easily checks that $\sigma(y,x)(t)=\sigma(x,y)(1-t)$, so that $\sigma$ is a $G$-section as claimed.

Conversely, suppose we have a local $G$-section $\sigma:U\to X^I$ of $\pi$. Then define a $G$-homotopy $H:U\times I\to X\times X$ by setting
\[
H(x,y,t)=\big( \sigma(x,y)(t/2), \sigma(x,y)(1-t/2)\big)\qquad (x,y)\in U, \,t\in I.
\]
Again, it is easily checked that $H$ has the required properties.
\end{proof}

Let $SP^2(X)=(X\times X)/G$ denote the \emph{symmetric square} of $X$ (the orbit space of our involution on $X\times X$). Let $dX\subseteq SP^2(X)$ denote the image of the diagonal subspace $\Delta(X)\subseteq X\times X$ under the orbit projection map $\rho: X\times X\to SP^2(X)$.

 \begin{thm}\label{thm:lowerSP}
 Suppose there are classes $x_1,\ldots ,x_k\in H^*(SP^2(X))$ (with arbitrary coefficients) such that:
 \be
 \item $x_i$ restricts to zero in $H^*(dX)$ for $i=1,\ldots, k$;
 \item $0\neq x_1\cdots x_k \in H^*(SP^2(X))$.
 \ee
 Then $\TC^\Sigma(X)\ge k$.
 \end{thm}
 \begin{proof}
 Given that classes as in the statement exist, suppose for a contradiction that $\TC^\Sigma(X)<k$. Then we may cover $X\times X$ by invariant open subsets $U_1,\ldots , U_k$, on each of which $\pi$ admits a local $G$-section. By Lemma \ref{lem:sectioniffdiagonal} we have, for each $i=1,\ldots, k$, a $G$-homotopy $H_i:U_i\times I\to X\times X$ from the inclusion to a map with values in the diagonal subspace $\Delta X\subseteq X\times X$. Let $\bar{U_i}:=\rho(U_i)$ be the quotient space of $U_i$. Since $I$ is locally compact, the product $\bar{U_i}\times I$ has the quotient topology from $U_i\times I$, hence $H_i$ induces a homotopy $\bar{H_i}:\bar{U_i}\times I\to SP^2(X)$. It is clear that $\bar{H_i}$ is a homotopy from the inclusion $\bar{U_i}\hookrightarrow SP^2(X)$ to a map with values in $dX\subseteq SP^2(X)$.

  By our assumption (1) and the long exact cohomology sequence of the pair $(SP^2(X),\bar{U_i})$, it follows that each $x_i$ comes from a class $\tilde{x}_i\in H^*(SP^2(X),\bar{U_i})$. By naturality of cup products, $x_1\cdots x_k$ is therefore the image of a class $$\tilde{x}_1\cdots\tilde{x}_k\in H^*(SP^2(X),\bar{U_1}\cup\cdots\cup\bar{U_k})=H^*(SP^2(X),SP^2(X))=0.$$ Hence  $x_1\cdots x_k=0$, contradicting assumption (2).
  \end{proof}

\section{Monoidal symmetrized topological complexity}\label{sec:MonoidalSymmTC}

Iwase and Sakai \cite{IS1,IS2} defined a variant of topological complexity, which they called \emph{monoidal topological complexity} and denoted $\TC^M(X)$, in which the motion from $A$ to $A$ is required to be constant at $A$. Their paper \cite{IS1} included a proof that $\TC^M(X)=\TC(X)$ for all locally finite simplicial complexes $X$, which was later found to contain an error \cite{IS2}. The question of whether $\TC^M(X)=\TC(X)$ for all spaces $X$ has become known as the Iwase-Sakai conjecture, and remains unanswered (although recent progress has been made, see \cite{CCV,Dran}). Here we define a monoidal version of symmetrized topological complexity. Unlike in the ordinary case, it is not too hard to show that the monoidal and non-monoidal symmetrized topological complexities coincide for a large class of spaces.

Recall that $\Delta X\subseteq X\times X$ denotes the diagonal subspace. This is an invariant subspace which admits a canonical local $G$-section $c:\Delta X\to PX$ of $\pi:PX \to X\times X$, given by setting $c(x,x)$ to be the constant path at $x$.

\begin{defn}\label{def:MSTC}
The \emph{monoidal symmetrized topological complexity} of $X$, denoted $\TC^{M,\Sigma}(X)$, is the least integer $k$ such that $X\times X$ may be covered by invariant open sets $U_0,U_1,\ldots , U_k$, each of which contains $\Delta X$ and admits a local $G$-section $\sigma_i:U_i\to PX$ of $\pi$ such that $\sigma_i|_{\Delta X} =c$.
\end{defn}

\begin{thm}\label{thm:MSvsS}
Let $X$ be a paracompact ENR. Then
\[
\TC^{M,\Sigma}(X)=\TC^\Sigma(X).
\]
\end{thm}

Theorem \ref{thm:MSvsS} will form the basis of our calculation of $\TC^\Sigma(S^1)$ in the next section. Before giving its proof, we describe a cohomological lower bound for $\TC^{M,\Sigma}(X)$ using the symmetric square, analogous to Theorem \ref{thm:lowerSP}.

\begin{thm}\label{thm:lowerMonoidal}
Suppose there are relative classes $x_1,\ldots ,x_k\in H^*(SP^2(X),dX)$ (with arbitrary coefficients) such that
$$0\neq x_1\cdots x_k \in H^*(SP^2(X),dX).$$ Then $\TC^{M,\Sigma}(X)\ge k$.
 \end{thm}
 \begin{proof}
Following through the construction of Lemma \ref{lem:sectioniffdiagonal}, one sees that an open invariant set $U\subseteq X\times X$ containing $\Delta X$ admits a local $G$-section $\sigma:U\to PX$ of $\pi$ satisfying $\sigma|_{\Delta X}=c$ if, and only if, $\Delta X\subseteq U$ is a strong $G$-deformation retract.

Now suppose $\TC^{M,\Sigma}(X)<k$, as witnessed by a cover of $X\times X$ by invariant open sets $U_1,\ldots , U_k$ as above. For each $i=1,\ldots, k$, it follows as in the proof of Theorem \ref{thm:lowerSP} that the quotient $\bar{U_i}\subseteq SP^2(X)$ contains $dX$ as a strong deformation retract. Examining the long exact sequence of the triple
\[
\xymatrix{
\cdots \ar[r] & H^*(SP^2(X),\bar{U_i}) \ar[r] & H^*(SP^2(X),dX) \ar[r] & H^*(\bar{U_i},dX) \ar[r] & \cdots
}
\]
one sees that each $x_i$ in the statement comes from a class $\tilde{x}_i\in H^*(SP^2(X),\bar{U_i})$. Since the $\bar{U_i}$ cover $SP^2(X)$, we have $\tilde{x_1}\cdots\tilde{x_k}=0$. Hence by naturality $x_1\cdots x_k=0$, a contradiction.
\end{proof}

The rest of this section is devoted to proving Theorem \ref{thm:MSvsS}. We invite the reader to compare Section 2 of \cite{Dran}, where the equality $\TC^M(X)=\TC(X)$ is proved for $s$-connected simplicial complexes $X$ satisfying $(s + 1)(\TC(X) + 1) > \dim \,X + 1$.

 By Proposition \ref{prop:Gfibjoin}, when $X$ is paracompact we have $\TC^\Sigma(X)\le k$ if and only if the $(k+1)$-fold fibred join
\[
\pi_k:J^k_{X\times X}(PX)\to X\times X
\]
admits a $G$-section. Note that the canonical local $G$-section $c:\Delta X\to PX$ induces a local $G$-section $c_k:\Delta X\to J^k_{X\times X}(PX)$ given by
\[
c_k(x,x) = \left[c(x,x),\frac{1}{k+1},\ldots , c(x,x),\frac{1}{k+1}\right],
\]
where as above $c(x,x)$ is the constant path at $x$.

\begin{lem}\label{lem:GfibjoinMonoidal}
Let $X$ be a paracompact ENR, and let $\sigma:X\times X\to J^k_{X\times X}(PX)$ be a $G$-section of $\pi_k$. Then
\[
\sigma|_{\Delta X}\simeq c_k:\Delta X\to J^k_{X\times X}(PX).
\]
\end{lem}

\begin{proof}
Since $\sigma$ is $G$-equivariant, it must take fixed points to fixed points, and it follows that $\sigma(\Delta X)\subseteq J^k_{X\times X}(PX)^G=J^k_{X\times X}(PX^G)$. Hence both $\sigma|_{\Delta X}$ and $c_k$ are sections of the fibration
\[
\xymatrix{
J^k(\Omega X^G)\ar[r] & J^k_{X\times X}(PX^G) \ar[r]^-{\pi_k^G} & \Delta X
}
\]
obtained from $\pi_k$ by restricting to fixed points. The fibre $J^k(\Omega X^G)$ is a join of contractible spaces (as noted in the proof of Theorem \ref{thm:upper}), hence is contractible. Since $X$ is paracompact and locally contractible, it follows (from \cite[Proposition 2.2]{Pavesic}, for example) that $\pi_k^G$ is fibre-homotopically trivial. Hence, in particular, all of its sections are homotopic.
\end{proof}

\begin{proof}[Proof of Theorem \ref{thm:MSvsS}]
The inequality $\TC^{M,\Sigma}(X)\ge\TC^\Sigma(X)$ is obvious. Assume $\TC^\Sigma(X)\le k$, and let $\sigma:X\times X\to J^k_{X\times X}(PX)$ be a $G$-section of $\pi_k$. By Lemma \ref{lem:GfibjoinMonoidal}, there exists a homotopy $K:\Delta X\times I\to J_{X\times X}^k(PX)$ from $\sigma|_{\Delta X}$ to $c_k$, which may be regarded as a $G$-homotopy where $G$ acts trivially on $\Delta X$. Now we apply Proposition \ref{prop:relGHLP} to the square
\[
\xymatrix{
X\times X\times\{0\}\cup\Delta X\times I \ar[rr]^-{\sigma\cup K} \ar[d] & & J^k_{X\times X}(PX) \ar[d]^{\pi_k} \\
X\times X\times I \ar[rr]^-{H} \ar@{.>}[urr]^{F} & & X\times X
}
\]
where $H$ is the identity homotopy. Since $X$ is ENR, the diagonal $\Delta X\subseteq X\times X$ is a closed $G$-cofibration by Lemma \ref{cor:symLEC}, and $\pi_k$ is a $G$-fibration. We obtain a fibrewise $G$-homotopy $F:X\times X\times I\to J^k_{X\times X}(PX)$ from $\sigma$ to a $G$-section $\widetilde\sigma$ which equals $c_k$ on the diagonal.

Setting $U_i=\widetilde\sigma^{-1}t_i^{-1}(0,1]$ as in the proof of Proposition \ref{prop:Gfibjoin} gives a cover of $X\times X$ by invariant open sets $U_i$, each of which contains the diagonal and admits a local $G$-sections $\widetilde\sigma_i:U_i\to J^k_{X\times X}(PX)$ satisfying $\widetilde\sigma_i|_{\Delta X}=c$. Hence $\TC^{M,\Sigma}(X)\le k$, and we are done.
\end{proof}

\section{Symmetrized topological complexity of spheres}\label{sec:spheres}

In this section we calculate the symmetrized topological complexity of spheres. This is possible largely due to Nakaoka's calculations \cite{Nak56} of the mod $2$ cohomology ring of $SP^2(S^n)$, and a result of Morton \cite{Morton} which states that $SP^2(S^1)$ is homeomorphic to the M\"obius band.

\begin{figure}
\begin{tikzpicture}
\tikzset{->-/.style={decoration={
  markings,
  mark=at position .4 with {\arrow{>}},
  mark=at position .6 with {\arrow{>}}},postaction={decorate}}}

  \draw[thick,fill=gray!30](0,0) rectangle (3,2);
\draw[thick,->-](0,0) to (0,2);
\draw[thick,->-](3,2) to (3,0);
\draw[thick,red,smooth](0,.5) to [out=0,in=200](1.5,1) to [out=20,in=180](3,1.5);
\draw[thick,blue](0,1) to (3,1);
\end{tikzpicture}
\caption{Embedded curves in the M\"obius band $(M,\partial M)$, each representing the dual of a class in $H^1(M,\partial M;\Z/2)$ whose square is nonzero}
\label{fig:mobius}
\end{figure}
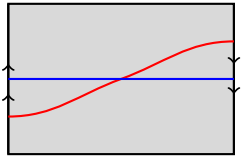

\begin{thm}\label{thm:spheres}
We have $\TC^\Sigma(S^n)=2$ for all $n\ge 1$.
\end{thm}
\begin{proof}
The upper bound $\TC^\Sigma(S^n)\le 2$ follows immediately from Theorem \ref{thm:upper}. When $n$ is even, the lower bound $\TC^\Sigma(S^n)\ge 2$ follows from $\TC^\Sigma(S^n)\ge \TC(S^n)$, as noted already in \cite{BGRT}.

For $n>1$ odd, we use Theorem \ref{thm:lowerSP} to obtain $\TC^\Sigma(S^n)\ge 2$. It suffices to find a class $x\in H^n(SP^2(S^n);\Z/2)$ whose image in $H^n(dX;\Z/2)$ is zero, and such that $0\neq x^2\in H^{2n}(SP^2(X);\Z/2)$. For this we refer to Nakaoka's work on the cohomology of cyclic products. In particular, Theorem 13.3 in \cite{Nak56} asserts the existence of a class $x=g_n(1) \in H^n(SP^2(S^n);\Z/2)$ whose square is nonzero if $n>1$, and Lemma 11.3 in \cite{Nak56} implies that $x$ restricts to zero in $H^n(dX;\Z/2)$.

When $n=1$, it follows from \cite{Morton} that there is a homeomorphism of pairs $(SP^2(S^1),dS^1)\cong (M,\partial M)$, where $M$ is the M\"obius band. Therefore there are no non-trivial cup products in $H^*(SP^2(S^1))$, and Theorem \ref{thm:lowerSP} is insufficient to conclude that $\TC^\Sigma(S^1)\ge 2$.

There is, however, a non-trivial cup product $$H^1(M,\partial M;\Z/2)\otimes H^1(M,\partial M;\Z/2)\to H^2(M,\partial M;\Z/2)$$
(Poincar\'e--Lefschetz dual to an intersection of transversely embedded curves, as depicted in Figure \ref{fig:mobius}) and so Theorems \ref{thm:MSvsS} and \ref{thm:lowerMonoidal} come to the rescue. We conclude that $\TC^\Sigma(S^1)=\TC^{M,\Sigma}(S^1)\ge 2$.

\end{proof}

\section{Symmetrized higher topological complexity}\label{sec:higher}

In this final section we state generalizations of our results to the \emph{symmetrized higher topological complexity}, introduced in \cite{BGRT}, whose definition is recalled below. Their proofs offer no new difficulties, hence are left as exercises for the interested reader.

Let $m\ge2$ be a natural number, and let $\Sigma_m$ denote the symmetric group on $m$ letters. Denote by $V_m$ the wedge sum of $m$ copies of the interval $I=[0,1]$, each with base point $0\in I$. (We avoid the notation $J_m$ for this space used in \cite{BGRT} and elsewhere, since it conflicts with our notation for the join.) The group $\Sigma_m$ acts on $V_m$ by permuting the intervals, and hence acts on the function space $X^{V_m}$, elements of which are $m$-tuples $\mathbf{\gamma}=(\gamma_1,\ldots , \gamma_m)$ of paths in $X$ with common initial point. The evaluation map
\[
e_m:X^{V_m}\to X^{m},\qquad e_m(\mathbf{\gamma})=\left(\gamma_1(1),\ldots , \gamma_m(1)\right)
\]
is a $\Sigma_m$-fibration, where $\Sigma_m$ acts on the $m$-fold Cartesian product $X^{m}$ by permuting the factors (compare Examples \ref{ex:interval} and \ref{ex:Gfibration}). We define the \emph{symmetrized $m$-th higher topological complexity} of a space $X$ to be
\[
\TC^{\Sigma}_m(X)=\secat_{\Sigma_m}(e_m),
\]
which obviously agrees with \cite[Definition 4.6]{BGRT}. Clearly, $\TC^\Sigma(X)=\TC^\Sigma_2(X)$.

The fibre of $e_m$ is $\Omega X^{m-1}$, and its various fixed point spaces are all either Cartesian powers of $\Omega X$, or are contractible. Hence Theorem \ref{thm:uppersecatG} applies to give the following upper bound, which generalizes Theorem \ref{thm:upper}.

\begin{thm}\label{thm:upperhigher}
Let $X$ be an $s$-connected polyhedron. Then
\[
\TC^\Sigma_m(X)< \frac{m\,\dim\,X+1}{s+1}.
\]
\end{thm}

The lower bound of Theorem \ref{thm:lowerSP} also generalizes easily. Let $SP^m(X)$ denote the $m$-th symmetric product of $X$, defined as the orbit space of $X^m$ under the stated $\Sigma_m$-action. Denote by $d_m X \subseteq SP^m(X)$ the image of the $m$-fold diagonal $\Delta_m X= \{(x,\ldots , x)\}\subseteq X^m$.

\begin{thm}\label{thm:lowerSPm}
 Suppose there are classes $x_1,\ldots ,x_k\in H^*(SP^m(X))$ (with arbitrary coefficients) such that:
 \be
 \item $x_i$ restricts to zero in $H^*(d_mX)$ for $i=1,\ldots, k$;
 \item $0\neq x_1\cdots x_k \in H^*(SP^m(X))$.
 \ee
 Then $\TC^\Sigma_m(X)\ge k$.
 \end{thm}

There is a \emph{monoidal symmetrized $m$-th higher topological complexity} $\TC^{M,\Sigma}_m(X)$, and the following generalization of Theorem \ref{thm:MSvsS} holds.

\begin{thm}\label{thm:MSvsShigher}
Let $X$ be a paracompact ENR. Then
\[
\TC^{M,\Sigma}_m(X)=\TC^\Sigma_m(X).
\]
\end{thm}
 The key observations here are that the $m$-fold iterated diagonal $\triangle_m:X\to X^m$ is a $\Sigma_m$-cofibration (by Proposition \ref{prop:GLEC}), and that the $\Sigma_m$-fixed points of the fibre of $e_m$ are homeomorphic to $P_0X$, just as in the case $m=2$.

 \begin{cor}\label{cor:lowerMonoidalhigher}
Suppose there are relative classes $x_1,\ldots ,x_k\in H^*(SP^m(X),d_mX)$ (with arbitrary coefficients) such that
$$0\neq x_1\cdots x_k \in H^*(SP^m(X),d_m X).$$ Then $\TC^{M,\Sigma}_m(X)\ge k$.
 \end{cor}

Finally, we discuss the symmetrized higher topological complexity of spheres. By Theorem \ref{thm:upperhigher} we have $\TC^\Sigma_m(S^n)\le m$ for all $m\ge2, n\ge1$. By \cite[\S 4]{Rudyak} we have
\[
\TC_m(S^n) = \left\{\begin{array}{ll} m & n \mbox{ even}, \\ m-1 & n \mbox{ odd}. \end{array}\right.
\]
Since clearly $\TC^\Sigma_m(X)\ge \TC_m(X)$ for any space $X$, it follows that $\TC^\Sigma_m(S^n)=m$ for $n$ even. One might expect that $\TC^{\Sigma}_m(S^n)=m$ for odd $n$ also, by analogy with the case $m=2$. Unlike the case $m=2$ however, this does not seem follow directly from the above lower bounds and Nakaoka's results on the cohomology of symmetric products of spheres.


\begin{thebibliography}{99}


\bibitem{BGRT} I. Basabe, J. Gonz\'alez, Y. B. Rudyak, D. Tamaki, \emph{Higher topological complexity and its symmetrization}, Algebr. Geom. Topol. {\bf 14} (2014), no. 4, 2103–-2124.


\bibitem{Bredon1} G. E. Bredon, \emph{Equivariant cohomology theories}, Bull. Amer. Math. Soc. {\bf 73} (1967), 266-–268.

\bibitem{Bredon2} G. E. Bredon, \emph{Equivariant cohomology theories}, Lecture Notes in Mathematics {\bf 34}, Springer-Verlag, Berlin-New York 1967.


\bibitem{CCV} J. G. Carrasquel-Vera, J. M. García-Calcines, L. Vandembroucq, \emph{Relative category and monoidal topological complexity}, Topology Appl. {\bf 171} (2014), 41-–53.


\bibitem{CG} H. Colman, M. Grant, \emph{Equivariant topological complexity}, Algebr.\ Geom.\ Topol. {\bf 12} (2012), no. 4, 2299-–2316.

\bibitem{Davis} D. Davis, \emph{The symmetric topological complexity of the circle}, New York J.\ Math. {\bf 23} (2017), 593–-602.

\bibitem{Dran} A. Dranishnikov, \emph{Topological complexity of wedges and covering maps}, Proc. Amer. Math. Soc. {\bf 142} (2014), no. 12, 4365–-4376.


\bibitem{EilenbergSteenrod} S. Eilenberg, N. Steenrod, \emph{Foundations of Algebraic Topology}, Princeton, Princeton University Press, 1952.


\bibitem{Far03} M. Farber, \emph{Topological complexity of motion planning}, Discrete Comput. Geom. {\bf 29} (2003), no. 2, 211–-221.

\bibitem{Far04} M. Farber, \emph{Instabilities of robot motion}, Topology Appl. {\bf 140} (2004), no. 2-3, 245-–266.

\bibitem{Far06} M. Farber, \emph{Topology of robot motion planning}, Morse theoretic methods in nonlinear analysis and
in symplectic topology, NATO Sci. Ser. II Math. Phys. Chem., vol. 217, Springer, Dordrecht,
2006, pp. 185–-230.

\bibitem{FG06} M. Farber, M. Grant, \emph{Symmetric motion planning},
Contemp. Math. {\bf 438} (2007), 85-–104.

\bibitem{FTY} M. Farber, S. Tabachnikov, S. Yuzvinsky, \emph{Topological robotics: motion planning in projective spaces}, Int.\ Math.\ Res.\ Not.\ {\bf 34} (2003), 1853-–1870.


\bibitem{Gonzalez} J. Gonz\'alez, \emph{Symmetric bi-skew maps and
symmetrized motion planning in projective spaces}, preprint. \href{https://arxiv.org/abs/1702.05457}{\texttt{arXiv:1702.05457}}

\bibitem{GonLan} J. Gonz\'alez, P. Landweber, \emph{Symmetric topological complexity of projective and lens spaces}, Algebr.\ Geom.\ Topol.\ {\bf 9} (2009), no. 1, 473–-494.




\bibitem{IS1} N. Iwase, M. Sakai, \emph{Topological complexity is a fibrewise L-S category}, Topology Appl. {\bf 157} (2010), no. 1, 10–-21.

\bibitem{IS2} N. Iwase, M. Sakai, \emph{Erratum to ``Topological complexity is a fibrewise L-S category" [Topology Appl. 157 (1) (2010) 10–21] }, Topology Appl. {\bf 159} (2012), no. 10-11, 2810–-2813.

\bibitem{James} I. M. James, \emph{On category, in the sense of Lusternik-Schnirelmann}, Topology {\bf 17} (1978), no. 4, 331–-348.


\bibitem{Lueck} W. L\"uck, \emph{Transformation groups and algebraic K-theory}, Lecture Notes in Mathematics, {\bf 1408} Springer-Verlag, Berlin, 1989.

\bibitem{Morton} H. R. Morton, \emph{Symmetric products of the circle}, Proc.\ Cambridge Philos.\ Soc. {\bf 63} (1967), 349-–352.

\bibitem{MukMuk} A. Mukherjee, G. Mukherjee, \emph{Bredon-Illman cohomology with local coefficients}, Quart.\ J.\ Math.\ Oxford Ser. (2) {\bf 47} (1996), no. 186, 199–-219.

\bibitem{Nak56} M. Nakaoka, \emph{Cohomology theory of a complex with a transformation of prime period and its applications}, J.\ Inst.\ Polytech.\ Osaka City Univ.\ Ser.\ A. {\bf 7} (1956), 51–-102.

\bibitem{Nak57} M. Nakaoka, \emph{Cohomology mod p of the p-fold symmetric products of spheres},
J.\ Math.\ Soc.\ Japan {\bf 9} (1957), 417-–427.

\bibitem{Pavesic} P. Pave\v{s}i\'c, \emph{A note on trivial fibrations}, Glas.\ Mat.\ Ser.\ III {\bf 46(66)} (2011), no. 2, 513–-519.

\bibitem{Rudyak} Y. B. Rudyak, \emph{On higher analogs of topological complexity}, Topology Appl.\ {\bf 157} (2010), no. 5, 916–-920.

\bibitem{Schwarz} A. Schwarz. \emph{The genus of a fiber space}, A.M.S. Transl. \textbf{55} (1966), 49--140.

\bibitem{Spa} E. H. Spanier, \emph{Algebraic topology}, McGraw-Hill, 1966.


\bibitem{Strom} A. Str{\o}m, \emph{Note on cofibrations}, Math.\ Scand.\ {\bf 19} (1966), 11-–14.



\end{thebibliography}
\end{document}